\documentclass[a4paper,10 pt,reqno]{amsart}
\usepackage[latin1]{inputenc}
\usepackage[english]{babel}
\usepackage{amsmath, amsthm, amssymb, amsopn, amsfonts, amstext, stmaryrd, enumerate, color, mathtools, hyperref, mathrsfs, enumitem, url}
\usepackage[left=2.65cm,right=2.65cm,top=2.7cm,bottom=2.7cm]{geometry}
\numberwithin{equation}{section} 
\usepackage{dsfont}
\usepackage{cite}

\hypersetup{colorlinks=true, 	linkcolor=blue!75!black, 	citecolor=red!50!black, 	urlcolor=green!50!black, 	linktoc=all }

\usepackage{colortbl}
\usepackage{tikz}

\newcommand{\N}{\mathbb{N}}

\newcommand{\R}{\mathbb{R}}

\renewcommand{\S}{\mathbb{S}^n}

\newcommand{\vertiii}[1]{{\left\vert\kern-0.25ex\left\vert\kern-0.25ex\left\vert #1
    \right\vert\kern-0.25ex\right\vert\kern-0.25ex\right\vert}}

\newcommand{\de}{\delta}

\newcommand{\Om}{\Omega}

\renewcommand{\epsilon} {\varepsilon}

\def\XXint#1#2#3{{\setbox0=\hbox{$#1{#2#3}{\int}$ }
\vcenter{\hbox{$#2#3$ }}\kern-.6\wd0}}

\theoremstyle{plain}
\newtheorem{theorem}{Theorem}[section]

\newtheorem{lemma}[theorem]{Lemma}

\theoremstyle{definition}

\theoremstyle{remark}
\newtheorem{remark}[]{Remark}

\renewcommand{\leq}{\leqslant}

\renewcommand{\geq}{\geqslant}

\begin{document}

\title[Small spheres with prescribed nonconstant mean curvature]{Small spheres with  prescribed nonconstant \\ mean curvature in Riemannian manifolds}

\author[A. Enciso, A. J. Fern\'andez and D. Peralta-Salas]{Alberto Enciso, Antonio J. Fern\'andez and Daniel Peralta-Salas}

\address{
\newline
\textbf{{\small Alberto Enciso}} \vspace{0.1cm}
\vspace{0.1cm}
\newline \indent Instituto de Ciencias Matem\'aticas, Consejo Superior de Investigaciones Cient\'ficas, 28049 Madrid, Spain}
\email{aenciso@icmat.es}

\address{
\vspace{-0.25cm}
\newline
\textbf{{\small Antonio J. Fern\'andez}} 
\vspace{0.1cm}
\newline \indent Departamento de Matem\'aticas, Universidad Aut\'onoma de Madrid, 28049 Madrid, Spain}
\email{antonioj.fernandez@uam.es}

\address{
\vspace{-0.25cm}
\newline
\textbf{{\small Daniel Peralta-Salas}} \vspace{0.1cm}
\vspace{0.1cm}
\newline \indent Instituto de Ciencias Matem\'aticas, Consejo Superior de Investigaciones Cient\'ficas, 28049 Madrid, Spain}
\email{dperalta@icmat.es}

\maketitle
$ $

\vspace{-0.5cm}

\begin{abstract}
Given a function $f$ on a smooth Riemannian manifold without boundary, we prove that if $p \in M$ is a non-degenerate critical point of $f$, then a  neighborhood  of $p$ contains a foliation by spheres with mean curvature proportional to $f$. This foliation is essentially unique. The nondegeneracy assumption can be substantially relaxed, at the expense of losing the property that the family of spheres with prescribed mean curvature defines a foliation.

\bigbreak
\noindent {\sc Keywords:} Mean curvature, Lyapunov--Schmidt, Brouwer degree, Riemannian manifold, foliation.

\noindent {\sc 2020 MSC:} 53A10, 53C42, 58J05, 58J55.
  
\end{abstract}

\section{Introduction and main results}

In this paper we are concerned with hypersurfaces with prescribed mean curvature in a Riemannian manifold without boundary $(M,g)$ of dimension $n+1$, with $n\geq1$. That is, given a positive function $f:M\to\R$, we are interested in the existence of smooth hypersurfaces $S\subset M$ whose mean curvature~$H$ coincides with~$f|_S$, possibly up to an unspecified multiplicative constant. For short, we will say that these hypersurfaces have prescribed mean curvature.

In the case of constant mean curvature (CMC) hypersurfaces, which corresponds to prescribing $f:= 1$, this question has attracted considerable attention. Techniques developed by White~\cite{White} show that, for a generic set of Riemannian metrics~$g$ on~$M$, the set of CMC hypersurfaces is a one-dimensional manifold, possibly empty or with infinitely many connected components. Regarding hypersurfaces diffeomorphic to a sphere, a classical result of Ye \cite{Ye1991}, which has been extended to several other geometric contexts by various authors~\cite{LMS,MePe2022,MMP}, asserts that~ if the scalar curvature of $(M,g)$ has a nondegenerate critical point~$p$, then there are many CMC spheres, which actually define a foliation on a (punctured) neighborhood of~$p$. For comparison, we shall next state a precise version of this fact:

\begin{theorem}[Ye, 1991] \label{T.Ye}
If $p \in M$ is a non-degenerate critical point of the scalar curvature of~$M$, then a punctured neighborhood of~$p$ contains a smooth foliation  $\mathscr{F}:= \{S_r : 0 < r < \delta\}$ by CMC spheres. The mean curvature of each~$S_r$ is $H = \overline{c}/r$ for some constant $\overline{c}> 0$.
\end{theorem}

Let us briefly turn to the study of overdetermined boundary value problems. This is a major topic in PDEs that appears in a number of mathematical and physical contexts such as fluid mechanics, spectral theory and isoperimetric inequalities, and where analogies with the theory of CMC hypersurfaces have played a major role in recent years~\cite{Ruiz17, Ruiz20, Sic09, Sic13, Sic15, Traizet, Ruiz22, delPino15}. The simplest setting for these problems is as follows. In Euclidean space~$\R^n$, one typically looks for solutions to a semilinear equation  in a domain~$\Om\subset\R^n$ for which one can prescribe both Dirichlet and Neumann boundary conditions:
\begin{align*}
	\Delta u + \lambda F(u)=0\qquad &\text{in }\Om\,,\\
u= 0\quad \text{and}\quad \partial_N u=-\bar c\qquad &\text{on }\partial\Om\,.
\end{align*}
Here $\lambda,\bar c$ are unspecified constants and $F$ is a reasonably well behaved function. A central result due to Serrin~\cite{Serrin} is that, by a clever adaptation of Alexandrov's moving plane method, positive solutions to a problem of this form on a smooth bounded domain must be radially symmetric, so the domain~$\Om$ must be a ball. The analogies go well beyond this fact. To mention just a few results, note that similar results can be proven on the hyperbolic space and on the round half-sphere~\cite{KP}, and that, despite the existence of various rigidity results~\cite{BCN, FV10, Ruiz17, Reichel , WW},  nontrivial solutions do exist on certain unbounded domains~\cite{Ruiz17, Ruiz20, Traizet} and on bounded domains when one drops the positivity condition~\cite{Ruiz22}.

In a general Riemannian manifold $(M,g)$ as above, and in particular when one considers position-dependent nonlinearities or nonconstant boundary data, our understanding of the problem is much more limited. In this context, one can take a well behaved nonlinearity $F:M\times\R\to\R$, boundary data $f_0,f_1:M\to\R$, and look at the problem 
\begin{align}\label{OBVP}
\begin{split}	
\Delta_g u + \lambda F(\cdot,u)=0\qquad &\text{in }\Om\,,\\
u= f_0\quad \text{and}\quad \partial_N u=-\bar c f_1\qquad &\text{on }\partial\Om\,.
\end{split}
\end{align}
on domains $\Om\subset M$. The first results in this direction~\cite{Sic09,Sic15} concern the case $F:=1$ (that is, the torsion problem on the manifold) with constant boundary conditions $(f_0,f_1):=(0,1)$, and ensure the existence of positive solutions to the overdetermined boundary problem on suitably deformed small geodesic balls centered at critical points of the scalar curvature under suitable technical conditions. For fairly general nonlinearities~$F:M\times\R\to\R$ and constant boundary conditions $(f_0,f_1):=(0,1)$, similar results were established in~\cite{AIM19}. These theorems are strongly motivated by the result of Ye on CMC hypersurfaces that we stated above. 

In the case of nonconstant boundary data, that is, for arbitrary functions $f_0$ and $f_1>0$, one can prove a similar result, although the strategy of the proof is surprisingly different. More precisely~\cite{APDE}, if $p$ is a non-degenerate critical point of the Dirichlet datum~$f_0$, under suitable technical hypotheses one can construct positive solutions to the overdetermined problem~\eqref{OBVP} on domains that are small, slightly perturbed balls centered at~$p$. The fact that there exist solutions to overdetermined problems with position-dependent coefficients and nonconstant data is not just a curiosity; in fact, these solutions have found applications in the study of incompressible fluids~\cite{ARMA21}.

Just as the results of~\cite{AIM19} can be seen as an analog of Ye's geometric theorem in the theory of overdetermined boundary problems, our goal in this paper is to close the circle by presenting an analog of the existence results for overdetermined problems with nonconstant data in the context of spheres with prescribed mean curvature. Specifically, we aim to prove the following result:

\begin{theorem} \label{mainintro}
Consider a positive function $f \in C^{m,\alpha}(M)$ with $m \geq 3$ and $\alpha \in (0,1)$. If $p \in M$ is a non-degenerate critical point of $f$, then a punctured neighborhood of~$p$ admits a $C^{m,\alpha}$ foliation  $\mathscr{F}:= \{S_r : 0 < r < \delta\}$ by spheres, of class $C^{m+2,\alpha}$, whose mean curvature is proportional to~$f$. Furthermore, the mean curvature of each~$S_r$ is $H = \overline{c} f/r$ for some constant $\overline{c}> 0$ and the foliation is smooth if $f \in C^{\infty}(M)$. 
\end{theorem}

We also show that, if $p$ is not a critical point of $f$, then such foliation cannot exist. Also, as happens in the classical setting considered in~\cite{Ye1991}, the foliation obtained in Theorem \ref{mainintro} is essentially unique. A precise statement of these facts is given in Theorem~\ref{uniquenessAndPpoint}.

Note that, contrary to what happens in the setting of CMC spheres, the critical points that play a role in the theorem are not those of the scalar curvature of the manifold, but those of the function~$f$ which describes the mean curvature. It is clear that this foliation is usually not given by the level sets of~$f$: if it were the case, the hypersurfaces $S_r$ would have constant mean curvature, so it is classical~\cite{Ye1991} that $p$ would then necessarily be a critical point of the scalar curvature of~$(M,g)$. It is also worth pointing out that the somewhat anisotropic regularity of the foliation, that is, the fact that the leaves of the foliation have two derivatives more that the foliation itself, is a consequence of the additional regularity gained by means of the second order elliptic equation relating the function~$f$ and the hypersurface. In the normal direction, no such effect is to be expected.

As in other papers in the literature~\cite{PX,Sic15}, it is possible to relax the assumption that there exists a non-degenerate critical point. The price one has to pay is that the spheres we obtain do not necessarily define a foliation, and that one cannot prove uniqueness. The way we have chosen to present this result is by considering a mean curvature function~$f$ with a critical point that is isolated but possibly degenerate. Our technical hypothesis is then that the index of the gradient field defined by~$f$ at this point must be nonzero, which eventually enables us to replace the implicit function theorem in the proofs by topological degree arguments.

\smallbreak 
\begin{theorem} \label{mainintro-index}
Consider a positive function $f \in C^{m,\alpha}(M)$ with $m \geq 3$ and $\alpha \in (0,1)$. If $p \in M$ is an isolated critical point of $f$ and the index of the gradient field $\nabla^g f$ at~$p$ is nonzero, then in a neighborhood of~$p$ there is a family  of spheres of class~$C^{m+2,\alpha}$, $ \{S_r : 0 < r < \delta\}$, such that the mean curvature of each~$S_r$ is  $H=\overline{c} f/r$ for some constant $\overline{c}> 0$. The spheres are smooth if $f\in C^\infty(M)$.
\end{theorem}


The paper is organized as follows. In Section~\ref{S.prelim}, we start by establishing preliminary geometric results and by recalling certain asymptotic expansions worked out by Ye that are useful in subsequent developments. For the benefit of the reader, we have maintained Ye's notation throughout the paper. The proof of Theorem~\ref{mainintro}, which is the main result of the paper, is presented in Section~\ref{sect3}. While the calculations are necessarily very different, the philosophy of the proof closely follows that of~\cite{APDE}. In Section \ref{sect3}, we also show that the foliation obtained in Theorem \ref{mainintro} is unique within the class of foliations we are considering and that, if $p$ is not a critical point of $f$, such foliation cannot exist. To conclude, in Section~\ref{sect4}, we present the proof of Theorem~\ref{mainintro-index}.

\section{Preliminaries}
\label{S.prelim}

In this section, we introduce some notation and establish some technical results that will be useful later on. We try to follow the notation used in~\cite{Ye1991} and~\cite[Section 2]{MePe2022} as closely as possible.
  
We denote the ball centered at $x_0\in\R^{n+1}$ of radius~$r$ by $\mathbb{B}_r(x_0) := \{x \in \R^{n+1}: |x-x_0| < r \}$. We also set $\mathbb{B}_r := \mathbb{B}_r(0)$. Here, $|\cdot|$ is the standard Euclidean norm in $\R^{n+1}$. Also, for $r > 0$, $\alpha_r: \R^{n+1} \to \R^{n+1}$ denotes the  homothety $\alpha_r(x):=rx$.
  
As in the Introduction, let $(M,g)$ be a smooth Riemannian manifold without boundary of dimension $n+1$, with $n \geq 1$, and let $f: M \to \R$ be a positive function of class $C^{m,\alpha}$ with $m \in \N$, $m \geq 3$, and $\alpha \in (0,1)$.  Given a point $p\in M$, let $R_p$ be the injectivity radius of~$(M,g)$ at the point~$p$ and let $r_p:=R_p/8$. Now, take vector fields $\{e_1, \ldots, e_{n+1}\}$ that are an orthonormal basis of $T_{p'}M$ for all $p'$ in a neighborhood of~$p$ and define the map
\begin{equation} \label{ctau}
c: \mathbb{B}_{R_p} \to M, \qquad \tau \mapsto {\rm{exp}}_p( \tau^i e_i),
\end{equation}
(which depends on $e_i(p)$ but is independent of the behavior of the vector fields $e_i$ away from~$p$). Here, and in the rest of the paper, we use the Einstein summation convention. 
This map determines a Riemannian normal coordinate system centered at $p$. For each $\tau \in \mathbb{B}_{r_p}$, we then define the map
\begin{equation} \label{phitau}
\varphi_{\tau}: \mathbb{B}_{2r_p} \to M, \qquad x \mapsto {\rm{exp}}_{c(\tau)}(x^i e_i^{\tau}),
\end{equation}
where $e_i^{\tau}$ is the parallel transport of $e_i(p)$ from~$p$ to the point~$c(\tau)$ along the geodesic $c(t\tau)|_{0\leq t \leq 1}$.  For future reference, let us record that, for each $\tau \in \mathbb{B}_{r_p}$, $\varphi_{\tau}$ gives rise to a Riemannian normal coordinate system centered at $c(\tau)$. We denote by $g_{ij}^{\tau}$ the coefficients of the metric tensor $g$ in this coordinate system. Similarly, we denote by $R_{ij}^{\tau}$ the coefficients of the Ricci tensor in the same coordinate system. Let us also note that 
$$
\varphi_0(x) = {\rm{exp}}_p(x^i e_i) = c(x), \quad \textup{for all } x \in \mathbb{B}_{2r_p} \subset \mathbb{B}_{R_p}.
$$

Also, throughout the paper, given a vector field $v$, $vf$ denotes the action of $v$ (as a first order differential operator) on a function~$f$.

Now, let $\nu$ be the inward unit normal to $\S$ and let $u \in C^2(\S)$. We denote by 
$$
S^{n}_{u}:= \{x + \nu(x) u(x): x \in \S\},
$$ 
the normal graph of $u$ over $\S$, and we define the surface 
\begin{equation} \label{Srtauu}
S_{r,\tau,u}:= \varphi_{\tau}(\alpha_r(S_{u}^n)).
\end{equation}
Note that $S_{r,\tau,0} = S_r(c(\tau))$ is the geodesic sphere of center $c(\tau) \in M$ and radius $r > 0$ and that there exists $\delta_0 > 0$ such that, if $\|u\|_{C^1(\S)} < \delta_0$, then $S_u^n$ is an embedded $C^2$ surface in $\R^{n+1}$. Moreover, the surface given in \eqref{Srtauu} is parametrized by the map
\begin{equation}
\varphi_{r,\tau,u}: \S \to S_{r,\tau,u},  \qquad x \mapsto {\rm{exp}}_{c(\tau)}\big(r(x^{i} + \nu^{i}(x) u(x)) e_i^{\tau} \big).
\end{equation}

Finally, for $0 < r < r_p$, $\tau \in \mathbb{B}_{r_p}$, $u \in C^2(\S)$ with $\|u\|_{C^1(\S)} < \delta_0$ and $x \in \S$, we define
\begin{equation} \label{Hdef1}
H[r,\tau,u](x) := \textup{ the inward mean curvature of }S_u^n\textup{ at } x + u(x) \nu(x) \textup{ w.r.t. } \overline{g} \textup{ on } \mathbb{B}_2,
\end{equation}
where $\overline{g} := r^{-2} \alpha_{r}^{*}(\varphi_\tau^*(g))$. Note that $\overline{g}$ extends smoothly to $r = 0$ with $\overline{g}\,|_{r=0} =: g_0$ the standard Euclidean metric. Also, by the scaling of the mean curvature, it follows that
\begin{equation} \label{Hdef2}
H[r,\tau,u](x) = r \; \textup{times the inward mean curvature of } S_{r,\tau,u} \textup{ at } \varphi_{r,\tau,u}(x) \textup{ w.r.t. } g . 
\end{equation}
We refer to \cite{Ye1991} for more details concerning \eqref{Hdef1}-\eqref{Hdef2}. 

  We need asymptotic expansions of the function $H$ and 
  \begin{equation}
F[r,\tau](x): \S \to \R, \qquad x \mapsto f(\varphi_{\tau}(rx)).
\end{equation}
These expansions play a key role in the proof of our main results.


\begin{lemma} \label{lemmaExpansion}
For $0 < r\ll1$, $\tau \in \mathbb{B}_{r_p}$, $u \in C^2(\S)$ with $r^2 \|u\|_{C^1(\S)} < \delta_0$,  $x \in \S$ and $\ell \in \{1, \ldots, n+1\}$,  one has:
\begin{align}
& \bullet\  H[r,\tau,r^2 u](x) = n + \Big( (\Delta_{\S}+n) u(x) -\frac13 R_{ij}^{\tau}(0) x^i x^j \Big)r^2 + O(r^3),\label{Hexpansion}  \\
& \bullet\  F[r,0](x) = f(p) + e_i f(p) x^i r + \frac12 e_i e_j f(p) x^i x^j r^2 + O(r^3), \label{Ftau0expansion} \\
 \label{Fpartialtau0expansion}
& \bullet\  \frac{\partial}{\partial \tau^\ell} \Big( F[r,\tau](x) \Big)\Big|_{\tau = 0} = e_\ell f(p) + e_\ell e_i f(p) x^i r + \frac12 e_\ell e_i e_j f(p) x^i x^j r^2 + O(r^3).  
\end{align}
\end{lemma}

\begin{proof}
First of all, note that \eqref{Hexpansion} immediately follows from \cite[Formula (1.19)]{Ye1991}. On the other hand, let us fix arbitrary $\tau \in \mathbb{B}_{r_p}$ and $x \in \S$. Taylor expanding $F[r,\tau](x)$ in $r$, we get that
\begin{equation*}
F[r,\tau](x)  = F[0,\tau](x) + \frac{\partial}{\partial r} \Big( F[r,\tau](x) \Big) \Big|_{r=0} r + \frac12 \frac{\partial^2}{\partial r^2}\Big( F[r,\tau](x) \Big) \Big|_{r=0} r^2 + O(r^3).
\end{equation*}
Now, let $\gamma_{x^i e_i^{\tau}}$ be the geodesic starting at $c(\tau)$ with initial velocity $x^i e_i^{\tau}$. Taking into account the properties of normal coordinates (see e.g. \cite[Proposition 5.24]{Lee2018}) and the definition of $F$, we get
\begin{align*}
& \frac{\partial}{\partial r} \Big( F[r,\tau](x) \Big) \Big|_{r=0} = \frac{\partial}{\partial r} \Big( f\big(\gamma_{x^i e_i^{\tau}}(r)\big) \Big) \Big|_{r=0} = \bigg( \Big( g^{\tau^{ij}} \Big|_{\gamma_{x^i e_i^{\tau}(r)}} e_i^{\tau} f \big(\gamma_{x^i e_i^{\tau}}(r)\big)\Big)e_j^{\tau} \cdot \dot{\gamma}_{x^i e_i^{\tau}}(r)  \bigg) \bigg|_{r= 0} \\
& \qquad = \bigg( \Big( g^{\tau^{ij}} \Big|_{\gamma_{x^i e_i^{\tau}(r)}} e_i^{\tau} f \big(\gamma_{x^i e_i^{\tau}}(r)\big)\Big)e_j^{\tau} \cdot \big(x^1, \ldots,x^{n+1}\big) \bigg) \bigg|_{r= 0} = e_i^{\tau} f(c(\tau)) x^i.
\end{align*}
Arguing on the same way, we get that
$$
\frac{\partial^2}{\partial r^2}\Big( F[r,\tau](x) \Big) \Big|_{r=0} = e_i^{\tau} e_j^{\tau} f(c(\tau)) x^i x^j.
$$
Hence, it follows that
\begin{equation} \label{Ftaugeneralexpansion}
F[r,\tau](x)  = f(c(\tau)) + e_i^{\tau} f(c(\tau))x^i r + \frac12 e_i^{\tau} e_j^{\tau} f(c(\tau)) x^i x^j r^2 + O(r^3).
\end{equation}
Taking into account the definition of $c(\tau)$, \eqref{Ftau0expansion} immediately follows from \eqref{Ftaugeneralexpansion}. Finally, differentiating \eqref{Ftaugeneralexpansion}  with respect to $\tau_\ell$ for an arbitrary but fixed $\ell \in \{1,\ldots, n+1\}$, and evaluating at $\tau = 0$, we get \eqref{Fpartialtau0expansion}.
\end{proof}

The following elementary lemma will be very useful in what follows. Here $\Gamma(z)$ denotes the Gamma function.

\begin{lemma}[\hspace{-0.005cm}\cite{Fo2001}] \label{lemmaFolland}
Let  $\alpha_1, \ldots, \alpha_{n+1} $ be nonnegative integers and set $\beta_j := \frac12(\alpha_j +1)$. Then\smallskip
$$
\int_{\S} (\omega^1)^{\alpha_1} (\omega^2)^{\alpha_2} \cdots (\omega^{n+1})^{\alpha_{n+1}}\, d\sigma(\omega) = \left\{
\begin{aligned}
\, & 0, \quad && \textup{ if some $\alpha_j$ is odd,} \\ 
& \frac{2\Gamma(\beta_1)\Gamma(\beta_2) \cdots \Gamma(\beta_{n+1})}{\Gamma(\beta_1 + \beta_2 + \cdots \beta_{n+1})}, && \textup{ if all $\alpha_j$ are even.}
\end{aligned}
\right.
$$
\end{lemma}

\section{The non-degenerate case: proof of Theorem \ref{mainintro}} \label{sect3}

The main goal of this section is to prove Theorem \ref{mainintro}. The heart of the proof is a Lyapunov--Schmidt reduction argument. 

Let us start by introducing some notation. It is well known that the first nonzero eigenvalue of the Laplacian on~$\S$ is~$n$, and that its corresponding eigenspace (that is, the kernel~$\mathcal K$ of the operator $\Delta_{\S} + n$) is spanned by the restriction to the sphere of the linear functions:
$$
\mathcal{K}:=  {\rm span} \{ x^i|_{\S}: 1 \leq i \leq n+1\} \subset C^{\infty}(\S) \,.
$$
We will omit the restriction to the sphere notationally when no confusion may arise. We henceforth denote by $\Pi$ the $L^2$-orthogonal projection onto $\mathcal{K}$. Therefore, given $h\in C^{0,\alpha}(\S)$, one has 
\begin{equation}
(\Pi h)(x)= \frac{n+1}{|\S|}\sum_{j=1}^{n+1} \left( \int_{\S} h(\omega) \omega^j d\sigma(\omega) \right) x^j.
\end{equation}
Also, let 
$
T : \mathcal{K} \to \R^{n+1}
$
be the isomorphism sending $x^i\in\mathcal K$ to the vector $\mathbf{e}_i$ of a Cartesian basis of $\R^{n+1}$, namely
$
T( a_i x^i ):= \sum_{i=1}^{n+1} a_i \mathbf{e}_i,
$
and set   
\[
\widetilde{\Pi}:= T \circ \Pi\,.
\]

Likewise, we define the complementary orthogonal projector $\Pi^{\perp}:= I-\Pi$ and define $\mathcal{K}^{\perp}_{\ell,\alpha}:= 
\Pi^{\perp}C^{\ell,\alpha}(\S)$, with $\ell \in \N \cup \{0\}$ and $\alpha \in (0,1)$. Therefore, $\mathcal{K}^{\perp}_{\ell,\alpha}$ is the space of $C^{\ell,\alpha}(\S)$ functions which are $L^2$-orthogonal to $\mathcal{K}$ and one has the decomposition
$$
C^{\ell,\alpha}(\S) = \mathcal{K} \oplus \mathcal{K}^{\perp}_{\ell,\alpha}.
$$

We start the proof of Theorem \ref{mainintro} with some preliminary lemmas.  Note that the first two will be also used to prove Theorem \ref{mainintro-index}. In the rest of this section, we fix $f: M \to \R$ to be a positive function of class $C^{m,\alpha}$ with $m \in \N$, $m \geq 3$, and $\alpha \in (0,1)$.


\begin{lemma} \label{lemmau0}
Let $p \in M$. There exists a unique $u_0 \in \mathcal{K}^{\perp}_{m+2,\alpha} \cap C^{\infty}(\S)$ solving
\begin{equation} \label{equationu0}
(\Delta_{\S}+n) u_0(x) = \frac13 \Big( R_{ij}^0(0) + \frac{3n}{2f(p)} e_i e_j f(p) \Big) x^i x^j , \quad \textup{ in } \S. 
\end{equation}
\end{lemma}

\begin{proof}
The result immediately follows if the right hand side in \eqref{equationu0} is $L^2$-orthogonal to $\mathcal{K}$. To show this is indeed the case, we just have to prove that the $L^2$-orthogonal projection of the right hand side in \eqref{equationu0} onto $\mathcal{K}$ is equal to zero. Note that
\begin{align*}
\Pi \left(  \frac13 \Big( R_{ij}^0(0) + \frac{3n}{2f(p)} e_i e_j f(p) \Big) x^i x^j \right) =  \frac{n+1}{3|\S|} \sum_{k=1}^{n+1} \left(  \Big( R_{ij}^0(0) + \frac{3n}{2f(p)} e_i e_j f(p) \Big) \int_{\S} \omega^{i} \omega^{j} \omega^{k} d\sigma(\omega) \right) x^k.
\end{align*}
On the other hand, by Lemma \ref{lemmaFolland}, we know that
$$
 \int_{\S} \omega^{i} \omega^{j} \omega^{k} d\sigma(\omega) = 0, \quad \textup{ for all } i,j,k \in \{1, \ldots, n+1\}.
$$
Thus, it follows that
$$
\Pi \left(  \frac13 \Big( R_{ij}^0(0) + \frac{3n}{2f(p)} e_i e_j f(p) \Big) x^i x^j \right) = 0.
$$
\end{proof}

\begin{lemma} \label{lemmaPiperp}
Let $p \in M$ be a critical point of $f$ and let  $u_0 \in \mathcal{K}^{\perp}_{m+2,\alpha}$  be as in Lemma \ref{lemmau0}. There exist a constant $r_0 \in (0,r_p)$ and a function $u \in C^{m,\alpha}([0,r_0)\times \mathbb{B}_{r_0}, \mathcal{K}^{\perp}_{m+2,\alpha})$ with $u(0,0) = u_0$ such that
\begin{equation}  \label{equationPiperp}
\Pi^{\perp} \left( \frac{1}{r^2} \Big(H[r,\tau,r^2 u(r,\tau)](x) - \frac{n}{f(p)} F[r,\tau](x) \Big) \right) = 0, \quad \textup{ in } \S. 
\end{equation}
\end{lemma}

\begin{proof}
First of all, taking into account Lemmas \ref{lemmaExpansion} and \ref{lemmau0}, we get that
$$
\begin{aligned}
& \Pi^{\perp} \left( \frac{1}{r^2} \Big(H[r,\tau,r^2 u](x) - \frac{n}{f(p)} F[r,\tau](x) \Big) \right)  \bigg|_{\substack{
r=0,\,\tau = 0 \\\hspace{-0.5cm} u=u_0}}  \\
& \hspace{6cm} = (\Delta_{\S} + n)u_0(x) - \frac13 \Big( R_{ij}^0(0) + \frac{3n}{2f(p)} e_i e_j f(p) \Big) x^i x^j = 0.
\end{aligned}
$$
On the other hand, it is immediate to check that the Fr\'echet derivative of the left hand side in \eqref{equationPiperp} with respect to $u$ satisfies 
$$
\begin{aligned}
& \frac{\partial}{\partial u} \left( \Pi^{\perp} \left( \frac{1}{r^2} \Big(H[r,\tau,r^2 u](x) - \frac{n}{f(p)} F[r,\tau](x) \Big) \right) \right) \Bigg|_{\substack{
r=0,\,\tau = 0 \\\hspace{-0.5cm} u=u_0}} = \big(\Delta_{\S} + n\big) \circ\Pi^{\perp},
\end{aligned}
$$
so it is a linear continuous invertible map from $\mathcal{K}^{\perp}_{m+2,\alpha}$ to  $\mathcal{K}^{\perp}_{m,\alpha}$. Also, note that
$$
\begin{aligned}
& \Pi^{\perp} \left( \frac{1}{r^2} \Big(H[r,\tau,r^2 u](x) - \frac{n}{f(p)} F[r,\tau](x) \Big) \right) \\
 & \quad =  \Pi^{\perp} \left( \frac{1}{r^2} \Big(H[r,\tau,r^2 u](x) - n \Big) \right) +  \Pi^{\perp} \left( \frac{1}{r^2} \Big(n - \frac{n}{f(p)} F[r,\tau](x) \Big) \right) =: G_1 [r,\tau,u](x) + G_2[r,\tau](x).
\end{aligned}
$$
We know that $G_1 \in C^{\infty}([0,r_p) \times \mathbb{B}_{r_p} \times \mathcal{K}_{m+2,\alpha}^{\perp},\, \mathcal{K}_{m,\alpha}^{\perp})$ and that $G_2 \in C^{m,\alpha}([0,r_p) \times \mathbb{B}_{r_p},\, \mathcal{K}_{m,\alpha}^{\perp})$. As the dependence in $u$ is smooth, we have uniform estimates in $C^{m,\alpha}$ in a neighborhood of $u_0 \in \mathcal{K}^{\perp}_{m+2,\alpha}$. Hence, the implicit function theorem in Banach spaces -- see for instance  \cite[Theorem 1.4]{AmMa2007} --  guarantees the existence of a constant $r_0 \in (0,r_p)$ and a function $u \in C^{m,\alpha}([0,r_0) \times \mathbb{B}_{r_0}, \mathcal{K}^{\perp}_{m+2,\alpha})$ with $u(0,0) = u_0$ such that 
\begin{equation*} 
\Pi^{\perp} \left( \frac{1}{r^2} \Big(H[r,\tau,r^2 u(r,\tau)](x) - \frac{n}{f(p)} F[r,\tau](x) \Big) \right) = 0, \quad \textup{ in } \S.
\end{equation*}
\end{proof}

\begin{lemma} \label{lemmaPitilde}
Let $p \in M$ be a non-degenerate critical point of $f$, let $u_0 \in \mathcal{K}^{\perp}_{m+2,\alpha}$ be as in Lemma \ref{lemmau0} and let $u \in C^{m,\alpha}([0,r_0) \times \mathbb{B}_{r_0}, \mathcal{K}^{\perp}_{m+2,\alpha})$ be as in Lemma \ref{lemmaPiperp}. There exist a constant $r_1 \in (0,r_0]$ and a function $\overline{\tau} \in C^{m,\alpha}([0,r_1),\R^{n+1})$ with $\overline{\tau}(0) = 0$ such that
\begin{equation} \label{equationPitilde}
\widetilde{\Pi} \left( \frac{1}{r} \Big(H[r,\overline{\tau}(r),r^2 u(r,\overline{\tau}(r))](x) - \frac{n}{f(p)} F[r,\overline{\tau}(r)](x) \Big) \right) = 0, \quad \textup{ in } \R^{n+1}. 
\end{equation}
\end{lemma}

\begin{proof}
First of all, note that, by Lemma \ref{lemmaFolland},
$$
\Pi\big( R_{ij}^{\tau}(0) x^i x^j \big) = 0 = \Pi \big( e_i e_j f(p) x^i x^j \big)  \quad \textup{ and  } \quad \Pi \big( e_i f(p) x^i \big) = e_i f(p) x^i.
$$
Also, observe that  $\widetilde{\Pi} \circ (\Delta_{\S} +n) = 0$. Combining \eqref{Hexpansion} and \eqref{Ftau0expansion} with this information and using that $p$ is a critical point of $f$, we get that
$$
\widetilde{\Pi} \left( \frac{1}{r} \Big(H[r,\tau,r^2 u(r,\tau)](x) - \frac{n}{f(p)} F[r,\tau](x) \Big) \right)  \bigg|_{\substack{r=0 \\ \tau = 0}} = -\frac{n}{f(p)} \sum_{i=1}^{n+1 } e_i f(p) \mathbf{e}_i = 0.
$$ 
On the other hand, taking into account \eqref{Fpartialtau0expansion}, we infer that, for all $\ell \in \{1,\ldots,n+1\}$,
$$
\frac{\partial}{\partial {\tau^\ell}} \left( \widetilde{\Pi} \left( \frac{1}{r} \Big(H[r,\tau,r^2 u(r,\tau)](x) - \frac{n}{f(p)} F[r,\tau](x) \Big) \right) \right)  \bigg|_{\substack{r=0 \\ \tau = 0}} = - \frac{n}{f(p)} \sum_{i=1}^{n+1} e_\ell e_i f(p) \mathbf{e}_i.
$$
Moreover, as a function of $\tau$, we have uniform estimates in $C^{m,\alpha}$ in a neighborhood of $\tau = 0$. Since $p$ is a non-degenerate critical point of $f$, using the implicit function theorem, we get a constant $r_1 \in (0,r_0]$ and a function $\overline{\tau} \in C^{m,\alpha}([0,r_1), \R^{n+1})$ with $\overline{\tau}(0) = 0$ such that
\begin{equation*} 
\widetilde{\Pi} \left( \frac{1}{r} \Big(H[r,\overline{\tau}(r),r^2 u(r,\overline{\tau}(r))](x) - \frac{n}{f(p)} F[r,\overline{\tau}(r)](x) \Big) \right) = 0, \quad \textup{ in } \R^{n+1}. 
\end{equation*}
\end{proof}

We now complete the proof of Theorem \ref{mainintro}. We provide here a more precise statement.



\begin{theorem} \label{mainresult}
Consider a positive function $f \in C^{m,\alpha}(M)$ with $m \geq 3$ and $\alpha \in (0,1)$. If $p \in M$ is a non-degenerate critical point of $f$, there exist a constant $\delta > 0$ and functions $\overline{\tau} \in C^{m,\alpha}([0,\delta],\, \R^{n+1})$ and $u \in C^{m,\alpha}([0,\delta],\, C^{m+2,\alpha}(\S))$, with $\overline{\tau}(0) = 0$, such that
\begin{equation} \label{equationmainresult}
H[r,\overline{\tau}(r), r^2 u(r)](x) = \frac{n}{f(p)} F[r,\overline{\tau}(r)](x), \quad \textup{ in } \S. 
\end{equation}
Hence, $\mathscr{F}:= \{S_r:= S_{r,\overline{\tau}(r),r^2 u(r)} : 0 < r < \delta\}$ is a family of prescribed mean curvature spheres of class $C^{m+2,\alpha}$ with $S_r$ having mean curvature $ H=(n/(rf(p))) f$. Moreover, $\mathscr{F}$ is a $C^{m,\alpha}$ foliation centered at $p$, and the foliation is smooth if $f\in C^\infty(M)$.
\end{theorem}

\begin{proof}
Let us start with an elementary observation which is the gist of the Lyapunov--Schmidt reduction method. If we find $\overline{\tau} = \overline{\tau}(r) \in \R^{n+1}$ with $\overline{\tau}(0) = 0$ and $u = u(r) \in C^{m+2,\alpha}(\S)$ solving the system \eqref{equationPiperp}--\eqref{equationPitilde}, for those $\overline{\tau}(r)$ and $u(r)$,  \eqref{equationmainresult} holds. Keeping this in mind, the existence part of the result immediately follows from Lemmas \ref{lemmaPiperp} and \ref{lemmaPitilde}. Hence, it just remains to show that the family 
$$
\mathscr{F}:= \big\{S_{r,\overline{\tau}(r), r^2 u(r)}: 0 < r \leq \delta\big\},
$$
is a $C^{m,\alpha}$ foliation centered at $p$. Arguing as in the proofs of Lemmas \ref{lemmaExpansion} and \ref{lemmaPitilde}, we infer that
$$
\overline{\tau}(r) = O(r^2), \quad \textup{ as } r \to 0^{+}.
$$
Having  this information at hand, the fact that $\mathscr{F}$ is a $C^{m,\alpha}$ foliation centered at $p$ immediately follows from the arguments in \cite[pages 390--391]{Ye1991} or \cite[page 7]{MePe2022}. The fact that the foliation is smooth when $f$ is smooth follows from a straightforward bootstrap argument. Indeed, if $f \in C^{\infty}(M)$, the function
$$
A(r,\tau,u) := \Pi^{\perp} \left( \frac{1}{r^2} \Big(H[r,\tau,r^2 u](x) - \frac{n}{f(p)} F[r,\tau](x) \Big) \right)
$$
is smooth in all its arguments for $r$, $|\tau|$ and $\|u- u_0\|_{C^{2,\alpha}(\S)}$ sufficiently small. Using then the representation formula given by the implicit function theorem in Banach spaces, we get that
$$
{\rm{d}} u (r,\tau) = - \partial_u A (r,\tau, u(r,\tau)) \circ \partial_{r,\tau} A(r,\tau, u(r,\tau)),
$$
for $r$ and $|\tau|$ sufficiently small. Since $A$ is smooth in all its arguments for $r$ and $|\tau|$ sufficiently small, the regularity of the right hand side is precisely the one of $u$. Hence, if for instance $u$ is $C^1$, we infer that ${\rm{d}}u$ is also $C^1$, which implies that $u$ is actually $C^2$. Iterating this process, we deduce that if $f \in C^{\infty}(M)$, then $u(r,\tau)$ is smooth for $r$ and $|\tau|$ sufficiently small. Arguing on the same way with the function 
$$
B(r,\tau):= \widetilde{\Pi} \left( \frac{1}{r} \Big(H[r,\tau,r^2 u(r,\tau)](x) - \frac{n}{f(p)} F[r,\tau](x) \Big) \right),
$$ 
we conclude that $\overline{\tau}$ is smooth for $r$ sufficiently small and the claimed regularity follows.  
\end{proof}

Concerning the uniqueness of the foliation $\mathscr{F}$ obtained in Theorem \ref{mainresult} and the need of $p$ being a critical point of $f$ for the existence of such foliation, we prove the following:

\begin{theorem} \label{uniquenessAndPpoint}
Consider a positive function $f \in C^{m,\alpha}(M)$ with $m \geq 3$ and $\alpha \in (0,1)$. \begin{itemize}
\item[i)] If there exists a foliation by spheres $\mathscr{F}_2 := \{S_r: 0 < r < \delta\}$ with $S_r$ having mean curvature $H = (n/(rf(p)))f$ that admits a parametrization of the form $S_r = S_{r,\tau^{*}(r), u^{*}(r)}$, with $\|u^{*}(r)\|_{C^{2,\alpha}(\S)} = O(r^{3/2})$ and $\tau^{*}(r) = o(1)$ as $r \to 0^{+}$, then $p$ is a critical point of $f$. 
  
\item[ii)] Let $p \in M$ be a non-degenerate critical point of $f$ and let $\mathscr{F}$ be the foliation obtained in Theorem \ref{mainresult}. If $\mathscr{F}_2$ is a foliation as in $i)$ and $U$ denotes the open set foliated by~$\mathscr{F}$, then then $\mathscr{F}_2|_U=\mathscr{F}$.
\end{itemize}
\end{theorem}

\begin{proof}
$i)$ First of all, observe that
\begin{equation} \label{identityTaustarUstar}
H[r,\tau^*(r), u^*(r)](x) = \frac{n}{f(p)} F[r,\tau^*(r)](x),  \qquad \textup{ in } \S.
\end{equation}
Then, combining the fact that $\|u^{*}(r)\|_{C^{2,\alpha}(\S)} = O(r^{3/2})$, as $r \to 0^{+}$, with Lemmas \ref{lemmaExpansion} and \ref{lemmaFolland} (see also \cite[formula (1.19)]{Ye1991}), we get that
\begin{align*}
0 & = \widetilde{\Pi} \left( \frac{1}{r} \Big( H[r,\tau^*(r), u^*(r)](x) - \frac{n}{f(p)} F[r,\tau^*(r)](x) \Big) \right) \\
& = - \frac{n}{f(p)} \sum_{i=1}^{n+1} e_i^{\tau^*(r)} f(c(\tau^*(r))) \mathbf{e}_i + O(r^2) \\
& = - \frac{n}{f(p)} \sum_{i=1}^{n+1} e_if(p) \mathbf{e}_i + \frac{n}{f(p)} \sum_{i=1}^{n+1} \left( e_i f(p) -  e_i^{\tau^*(r)} f(c(\tau^*(r))) \right) \mathbf{e}_i + O(r^2), \quad \textup{ as } r \to 0^{+},
\end{align*}
and the first part of the result immediately follows from the assumption $\tau^*(r) = o(1)$, as $r \to 0^{+}$. 

$ii)$ Let $p \in M$ be a nondegenerate critical point of $f$ and let $\mathscr{F}_2$ be as in $i)$. Following the proof of $i)$, we infer that
\begin{align*}
0 & = - \frac{n}{f(p)} \sum_{i=1}^{n+1} e_i^{\tau^*(r)} f(c(\tau^*(r))) \mathbf{e}_i + O(r^2) \\
& = - \frac{n}{f(p)} \sum_{i=1}^{n+1} e_j e_i f(p) (\tau^{*}(r))^j  \mathbf{e}_i + O(r^2) + O(|\tau^*(r)|^2), \quad \textup{ as } r \to 0^{+},
\end{align*}
and so, using the non-degeneracy assumption, we get that $\tau^*(r) = O(r^2)$ as $r \to 0^+$. Then, combining this size  with \eqref{identityTaustarUstar} and the fact that $p$ is a critical point of $f$, we get that
\begin{align*}
0 & = \Pi^{\perp}  \left( \frac{1}{r} \Big( H[r,\tau^*(r), u^*(r)](x) - \frac{n}{f(p)} F[r,\tau^*(r)](x) \Big) \right) \\
& = n + [(\Delta_{\S} + n)u^*(r)](x) - \frac13 R_{ij}^{\tau^*(r)}(0) \, x^i x^j r^2 \\
& \quad  - \frac{n}{f(p)} \left( f(c(\tau^*(r))) + \frac12 e_i^{\tau^*(r)} e_j^{\tau^*(r)} f(c(\tau^*(r))) x^i x^j r^2 \right) + O(r^3)\\
& = [(\Delta_{\S} + n) u^*(r)](x) - \frac13 \left( R_{ij}^{\tau^*(r)}(0) + \frac{3n}{2f(p)} e_i^{\tau^*(r)} e_j^{\tau^*(r)} f(c(\tau^*(r))) \right) x^i x^j r^2 \\
& \qquad - \frac{n}{f(p)} e_i f(p) (\tau^*(r))^i + O(|\tau^*(r)|^2) + O(r^3) \\
& =  [(\Delta_{\S} + n) u^*(r)](x) - \frac13 \left( R_{ij}^{\tau^*(r)}(0) + \frac{3n}{2f(p)} e_i^{\tau^*(r)} e_j^{\tau^*(r)} f(c(\tau^*(r))) \right) x^i x^j r^2 + O(r^3), \quad \textup{as } r \to 0^+.
\end{align*}
From standard elliptic estimates, we then infer that $\|u^*(r)\|_{C^{2,\alpha}(\S)} = O(r^2)$, as $r \to 0^+$. 

Now, we define $\overline{u}(r) := u^*(r) r^{-2}$ and, for $u_0 \in \mathcal{K}_{m+2,\alpha}^{\perp} \cap C^{\infty}(\S)$ as in Lemma \ref{lemmau0}, we get that
\begin{align*}
[(\Delta_{\S} + n )(\overline{u}(r) - u_0)](x) = \, & \frac{1}{3} \left(R_{ij}^{\tau^*(r)}(0) - R_{ij}^0(0) \right) x^i x^j \\
& + \frac{n}{2f(p)} \left( e_i^{\tau^*(r)} e_j^{\tau^*(r)} f(c(\tau^*(r))) - e_i e_j f(p) \right) x^i x^j + O(r), \quad \textup{as } r \to 0^{+}.
\end{align*}
Using again standard elliptic estimates, we infer that
$$
\lim_{r \to 0^+} \|\overline{u}(r) - u_0\|_{C^{2,\alpha}(\S)} = 0.
$$
Having at hand this convergence, the result follows from the uniqueness parts of the implicit function theorems used in Lemmas \ref{lemmaPiperp} and \ref{lemmaPitilde}.
\end{proof}

\section{The degenerate case: proof of Theorem \ref{mainintro-index}} \label{sect4}

At the expense of losing the regularity of the function $\overline{\tau}$ and the fact that the spheres we construct define a foliation, we can remove the non-degeneracy assumption from Theorem \ref{mainresult}. This section is devoted to prove Theorem~\ref{mainintro-index}, which makes this fact precise.  The proof of this result shares some elements with the proof of Theorem \ref{mainresult}. However, instead of using the implicit function theorem twice as before, we use a topological degree argument, which allows to remove the non-degeneracy assumption from Theorem \ref{mainresult}. As happens in~\cite{PX}, the reason for which we cannot conclude that we have a foliation with prescribed mean curvature is that the spheres can, in principle, intersect.

We refer to~\cite[Chapter 3]{AmMa2007} for the standard notation concerning the Brouwer degree we are going to use. Let us just remind here that, given a continuous vector field $X: \R^{n+1} \to \R^{n+1}$, its index at an isolated zero~$x_0 \in \R^{n+1}$ can be defined as
$
{\rm{ind}}(X, x_0) := \lim_{r \to 0^+} \deg \big( {X, \mathbb{B}_r(x_0), 0} \big).
$

\begin{theorem} \label{mainresult-index}
Consider a positive function $f \in C^{m,\alpha}(M)$ with $m \geq 3$ and $\alpha \in (0,1)$. If $p \in M$ is an isolated critical point of $f$ and the index of the gradient field $\nabla^g f$ at $p$ is nonzero, there exist a constant $r_{\star} > 0$, a function $u \in C^{m,\alpha}([0,r_{\star}) \times \mathbb{B}_{r_{\star}}, C^{m+2,\alpha}(\S))$, and another constant $\delta\in (0,r_\star)$ for which the following holds true: for all $r\in[0,\de]$, there exists some $\overline{\tau}(r) \in \mathbb{B}_{r_\star}$ such that
\begin{equation} \label{equationmainresul-index}
H[r,\overline{\tau}(r), r^2 u(r,\overline{\tau}(r))](x) = \frac{n}{f(p)} F[r,\overline{\tau}(r)](x), \quad \textup{ in } \S. 
\end{equation}
Hence, $\mathscr{F}:= \{S_r\equiv  S_{r,\overline{\tau}(r),r^2 u(r)}: 0 < r < \delta\}$ is a family of spheres of class $C^{m+2,\alpha}$ and $S_r$ has mean curvature $H=(n/(rf(p)))f$. The spheres are smooth if $f \in C^{\infty}(M)$. 
\end{theorem}

\begin{proof}
Let $u \in C^{m,\alpha}([0,r_0) \times \mathbb{B}_{r_0}, \mathcal{K}^{\perp}_{m+2,\alpha})$ be as in Lemma \ref{lemmaPiperp}. Arguing as in the proof of Theorem~\ref{mainresult}, the proof of this result is reduced to show the existence of $\delta \in (0,r_0)$ such that, for all $r \in [0,\delta]$, there exists $\overline{\tau}(r) \in \mathbb{B}_{r_0}$ satisfying
$$
\widetilde{\Pi} \left( \frac{1}{r} \Big(H[r,\overline{\tau}(r),r^2 u(r,\overline{\tau}(r))](x) - \frac{n}{f(p)} F[r,\overline{\tau}(r)](x) \Big) \right) = 0, \quad \textup{ in } \R^{n+1}. 
$$
We prove this is indeed the case using a topological degree argument. Let us introduce the map
\begin{equation}
G: [0,r_0) \times \mathbb{B}_{r_0} \to \R^{n+1}, \quad (r,\tau) \mapsto \widetilde{\Pi} \left( \frac{1}{r} \Big(H[r,\tau,r^2 u(r,\tau)](x) - \frac{n}{f(p)} F[r,\tau](x) \Big) \right).
\end{equation}
Since $p$ is a critical point of $f$ with $f(p) \neq 0$, combining Lemmas \ref{lemmaExpansion} and \ref{lemmaFolland}, we get that
$$
G(0,0) = -\frac{n}{f(p)} \sum_{i=1}^{n+1} e_i f(p) \mathbf{e}_i = 0. 
$$
Moreover, since $p$ is an isolated critical point of $f$, there exists $r_1 \in (0,r_0]$ such that
$$
G(0,\tau) = -\frac{n}{f(p)}\sum_{i=1}^{n+1} e_i^{\tau} f(c(\tau)) \mathbf{e}_i \neq 0, \quad \textup{ for all } \tau \in \overline{\mathbb{B}}_{r_1} \setminus \{0\}.
$$ 
Thus, it follows that
$
{\rm{ind}} (G(0, \cdot), 0) 
$
is well-defined. 

We now choose local coordinates $\tau = (\tau^1, \ldots, \tau^{n+1})$ such that the representation of $p$ in coordinates is $\tau = 0$ and $\partial_{\tau^i}|_{\tau = 0} = e_i$. Thus, we have $g_{ij}(0) = \delta_{ij}$ and $e_i^{\tau} = B_{ij} \partial_{\tau^j}$ with $B_{ij}(\tau) = \delta_{ij} + O(|\tau|)$. Also, we introduce the auxiliary function
$$
\widetilde{f}: \mathbb{B}_{r_0} \to \R, \qquad \tau \mapsto f(c(\tau)),
$$
and denote $D_{\tau} f(\tau) := (\partial_{\tau^1} \widetilde{f}(\tau), \ldots, \partial_{\tau^{n+1}} \widetilde{f}(\tau))$. Then, it is immediate to check that
\begin{equation} \label{G0homotopy}
G(0,\tau) = B(\tau) D_{\tau} f(\tau).
\end{equation}
On the other hand, the representation in coordinates of $\nabla^g f$ is  precisely
\begin{equation} \label{G1homotopy}
\nabla^g f = g^{-1}  D_{\tau} f,
\end{equation}
where the matrix $g^{-1} :=( g^{ij})$ is the inverse of the metric $g$.

Having at hand \eqref{G0homotopy}-\eqref{G1homotopy}, we introduce the map
\begin{equation*}
\Gamma: [0,1] \times \mathbb{B}_{r_0} \to \R^{n+1}, \qquad (s,\tau) \mapsto \left[ (1-s) B(\tau) + s g(\tau)^{-1} \right] D_{\tau} f(\tau).  
\end{equation*}
Since $ A_s:= (1-s) B + s g^{-1} = (1-s) I + s g^{-1} + O(|\tau|)$, we deduce there exists $r_2 \in (0,r_1]$ such that $\det(A_s(\tau)) \neq 0$ for all $(s,\tau) \in [0,1] \times \overline{\mathbb{B}}_{r_2}$. Hence, since $p$ is an isolated critical point of $f$, we infer there exists $r_3 \in (0,r_2]$ such that $\Gamma(s,\tau) \neq 0$ for all $(s,\tau) \in [0,1] \times \partial \mathbb{B}_{r_3}$ and so, that $\Gamma:[0,1] \times \mathbb{B}_{r_3} \to \R^{n+1}$ is an admissible homotopy. Moreover, observe that $\Gamma(0,\cdot) = G(0,\cdot)$ and $\Gamma(1,\cdot) = g^{-1} D_{\tau} f$. Then, by the homotopy invariance of the Brouwer degree and the excision property, we get that 
$$
\deg (G(0, \cdot), \mathbb{B}_{r}, 0) = \deg(g^{-1} D_{\tau} f, \mathbb{B}_r,0), \quad \textup{ for all } r \in (0,r_3].
$$
Thus, we conclude that
$$
{\rm{ind}}(G(0,\cdot),0) \neq 0. 
$$
The theorem then follows from, for instance, \cite[Remark 4.2.1 and Theorem 4.3.4]{AmAr2011}. 
\end{proof}

\section*{Acknowledgements} 

This work has received funding from the European Research Council (ERC) under the European Union's Horizon 2020 research and innovation programme through the grant agreement~862342 (A.E.\ and A.J.F.). It is partially supported by the grants CEX2019-000904-S, RED2022-134301-T and PID2022-136795NB-I00 (A.E. and D.P.-S.) funded by MCIN/AEI/10.13039/501100011033, and Ayudas Fundaci\'on BBVA a Proyectos de Investigaci\'on Cient\'ifica 2021 (D.P.-S.).

\bibliography{Bibliography}
\bibliographystyle{abbrv}
\vspace{0.2cm}

\end{document}